\documentclass[12pt]{article}

\usepackage{amsthm, amssymb, amstext}
\usepackage[fleqn]{amsmath}
\usepackage{amsthm}
\usepackage{todonotes}
\usepackage{comment}
\usepackage{esvect}

\def\A{\mathcal{A}}
\def\B{\mathcal{B}}
\def\C{\mathcal{C}}
\def\D{\mathcal{D}}

\def\F{\mathcal{F}}

\pgfmathsetmacro\Rad{3.8}
\pgfmathsetmacro\rad{3}
\pgfmathsetmacro\wid{.3}
\pgfmathsetmacro\smal{\rad-\wid}
\pgfmathsetmacro\larg{\rad+\wid}

\newtheorem{theorem}{Theorem}

\theoremstyle{definition}

\definecolor{gren}{RGB}{  0,140, 10}

\title{A simple proof of Talbot's theorem for intersecting separated sets}
\date{}
\author{Peter Borg\thanks{Department of Mathematics, Faculty of Science, University of Malta, Malta, email: \texttt{peter.borg@um.edu.mt}} \quad 
Carl Feghali\thanks{Computer Science Institute of Charles University, Prague, Czech Republic, email: \texttt{feghali.carl@gmail.com} }
}

\begin{document}   
\maketitle

\begin{abstract}

A subset $A$ of $[n] = \{1, \dots, n\}$ is \emph{$k$-separated} if, when the elements of $[n]$ are considered on a circle, between any two elements of $A$ there are at least $k$ elements of $[n]$ that are not in $A$. A family $\A$ of sets is \emph{intersecting} if every two sets in $\A$ intersect. We give a short and simple proof of a remarkable result of Talbot (2003), stating that if $n \geq (k + 1)r$ and $\A$ is an intersecting family of $k$-separated $r$-element subsets of $[n]$, then $|\A| \leq \binom{n - kr - 1}{r - 1}$. This bound is best possible.
 \end{abstract}

Let $[n]$ denote the standard $n$-element set $\{1, \dots, n\}$, and let $[n]^{(r)}$ denote the family of $r$-element subsets of $[n]$. Let $2^{[n]}$ denote the family of subsets of $[n]$, that is, $2^{[n]} = \bigcup_{r=0}^n [n]^{(r)}$. A family $\A \subseteq 2^{[n]}$ is said to be \emph{intersecting} if $A \cap B \neq \emptyset$ for every $A, B \in \A$. 

How large can an intersecting family $\A \subseteq [n]^{(r)}$ be? If $2r > n$, then any two sets in $[n]^{(r)}$ intersect, so $\A$ can be $[n]^{(r)}$ itself. 
The case $2r \leq n$ is much more difficult, and the solution is given by one of the most important theorems in extremal set theory, the Erd\H{o}s--Ko--Rado Theorem \cite{erdos}.

\begin{theorem}[Erd\H{o}s, Ko, Rado \cite{erdos}]\label{thm:ekr}
If $n$ and $r$ are positive integers, $n \geq 2r$, and $\A$ is an intersecting subfamily of $[n]^{(r)}$, 
then 
\[ |\A| \leq \binom{n - 1}{r - 1}, \]
and equality holds if, for some $x \in [n]$, $\A = \{A \in [n]^{(r)} \colon x \in A\}$. 
\end{theorem}

Many new proofs (see, for example, \cite{daykin, franklfuredi, katonaint, katonaekr}), extensions and generalizations of Theorem~\ref{thm:ekr} have appeared; we refer to \cite{deza, franklshifting, franklex, franklsurvey,  gerbner} for further details. One way of extending Theorem~\ref{thm:ekr} is by adding restrictions to the family under consideration. A subset $A = \{a_1, \dots , a_r\}$ of $[n]$ is said to be \emph{$k$-separated} if $a_{i+1} > a_i + k$ for each $i \in [r]$, where $a_{r+1} = a_1 + n$. The family of $k$-separated $r$-element subsets of $[n]$ is denoted by $[n]_k^{(r)}$. In a remarkable paper \cite{talbotcycle}, Talbot obtained the following analogue of Theorem \ref{thm:ekr} for $k$-separated sets, resolving a conjecture of Holroyd and Johnson \cite{holroyd1, holroyd2}.

\begin{theorem}\label{thm:main}
If $n$, $k$ and $r$ are positive integers, $n \geq (k + 1)r$, and $\A$ is an intersecting subfamily of $[n]^{(r)}_k$, then 
\[|\A| \leq \binom{n - kr - 1}{r - 1},\]
and the upper bound is attained if, for some $x \in [n]$, $\A = \{A \in [n]^{(r)}_k \colon x \in A\}$. 
\end{theorem}
Note that $[n]^{(r)}_k$ is empty if $n < (k+1)r$.

To prove Theorem \ref{thm:main}, Talbot introduced a shifting technique which, roughly speaking,  rotates anticlockwise the elements of the sets of
the intersecting family which are distinct from a specified element.The conjecture had attracted much interest, especially when it started to appear unlikely that an easy proof exists; see \cite{hiltontalbot} for a discussion. 

The purpose of this paper is to show that, surprisingly, a rather simple proof of Theorem \ref{thm:main} does exist. It also turns out that, in the same way that Hilton and Spencer \cite{hiltontalbot} obtained a generalization of Theorem \ref{thm:main} via an application of Talbot's  aforementioned shifting technique, one may also apply our shifting argument  instead and again significantly shorten and simplify their proof; we leave the details to the reader.  

For clarity of presentation, as in \cite{talbotcycle}, we first consider the case $k = 1$, and then we address the general case $k \geq 1$ by applying essentially the same argument. 

A $1$-separated subset $A$ of $[n]$ is also simply called a \emph{separated} set. Note that $A$ is separated if no two elements of $A$ are adjacent. (Phrasing this differently again, $A$ is separated if it is an independent set of the cycle $C_n$ whose vertices are labelled $1$ to $n$ clockwise.)

\begin{theorem}
If $n$ and $r$ are positive integers, $n \geq 2r$, and $\A$ is an intersecting subfamily of $[n]_1^{(r)}$, then 
\[|\A| \leq \binom{n - r - 1}{r - 1},\]
and the upper bound is attained if, for some $x \in [n]$, $\A = \{A \in [n]_1^{(r)} \colon x \in A\}$. 
\end{theorem}

\begin{proof}
We proceed by induction on $n$. The result is trivial if $n \leq 4$ or $r = 1$. Suppose $n \geq 5$ and $r \geq 2$. If $n = 2r$, then there are only two sets in $[n]^{(r)}_1$, and these are disjoint. Suppose $n \geq 2r+1$. Let $\A$ be an intersecting subfamily of $ [n]^{(r)}_1$. 

Let $f: [n]^{(r)}_1 \rightarrow [n]^{(r)}_1$ be the function such that, for each $A \in   [n]^{(r)}_1$, $f(A) = (A \setminus \{n\}) \cup \{n - 1\}$ if $n \in A$ and $n - 2 \not\in A$, and $f(A) = A$ otherwise. Let 
\[\A^* = \{f(A) \colon A \in \A\} \cup \{A \in \A \colon f(A) \in \A\}.\]
Clearly, $|\A^*| = |\A|$. For $\F \subseteq [n]^{(r)}_1$ and $i \in [n]$, let $\F(i) = \{F \setminus \{i\} \colon F \in \F, \, i \in F\}$. Let $\B = \{B \in \A^* \colon 1, \, n-1 \in B\}$, and let $\C = \B(n - 1) \cup \A^*(n)$. 

We claim that $\C$ is intersecting. 
Let $A, B \in \A^*(n)$. Suppose $A \cap B = \emptyset$. Then, $n - 2 \not\in A$ or $n - 2 \not\in B$. We may assume that $n - 2 \not\in A$. Let $A' = A \cup \{n-1\}$ and $B' = B \cup \{n\}$. By the definition of $f$, we have $A', B' \in \A$. However, we also have $A' \cap B' = \emptyset$, which contradicts the assumption that $\A$ is intersecting. Therefore, $\A^*(n)$ is intersecting. Since $\A$ is intersecting and each set in $\A$ is separated, it easily follows that $\C$ is intersecting, as claimed.  

We have $\C \subseteq [n - 2]_1^{(r - 1)}$. By the induction hypothesis, 
\[|\C| \leq \binom{n - r - 2}{r - 2}.\] 
 
Let $\A^*_{n} = \{A \in \A^* \colon n \in A\}$, let $\C^* = \A^*_{n} \cup \B$, and let $\D = \A^* \setminus \C^*$. Then, $\D \subseteq [n - 1]_1^{(r)}$. Since $\A$ is intersecting, $\D$ is intersecting. By the induction hypothesis, 
\[|\D| \leq \binom{n - r - 2}{r - 1}.\]

For each $A \in \A^*(n)$, $A$ is a separated set, and hence $1 \notin A$. Thus, $|\C| = |\B(n - 1)| + |\A^*(n)| = |\C^*|$. We have 
\[|\A| = |\A^*| = |\C^*| + |\D| \leq \binom{n - r - 2}{r - 2} + \binom{n - r - 2}{r - 1} \leq \binom{n - r - 1}{r - 1},\] 
as required. 
\end{proof}

\begin{proof}[Proof of Theorem \ref{thm:main}]
We proceed by induction on $n$. The result is trivial if $n \leq 4$ or $r = 1$. Suppose $n \geq 5$ and $r \geq 2$. If $n = (k + 1)r$, then there are only $k + 1$ sets in $[n]^{(r)}_k$, and these are pairwise disjoint. Suppose $n \geq (k + 1)r + 1$. Let $\A$ be an intersecting subfamily of $ [n]^{(r)}_k$. 

Let $f: [n]^{(r)}_k \rightarrow [n]^{(r)}_k$ be the function such that, for each $A \in [n]^{(r)}_k$, $f(A) = (A \setminus \{n\}) \cup \{n - 1\}$ if $n \in A$ and $n - k - 1 \not\in A$, and $f(A) = A$ otherwise. Let
\[\A^* = \{f(A) \colon A \in \A\} \cup \{A \in \A \colon f(A) \in \A\}.\]
Clearly, $|\A^*| = |\A|$. For $\F \subseteq [n]^{(r)}_k$ and $i \in [n]$, let $\F(i) = \{F \setminus \{i\} \colon F \in \F, \, i \in F\}$. For $i, j \in [n]$, let $\A^*_{i, j} = \{A \in \A^* \colon i, j \in A\}$. Let \[\C =  \A^*(n) \cup \bigcup_{i = 1}^k \A^*_{i, n - k - 1 + i}(n - k - 1 + i).\] 

We claim that $\C$ is intersecting. 
Let $A, B \in \A^*(n)$. Suppose $A \cap B = \emptyset$. Then, $n - k - 1 \not\in A$ or $n - k - 1 \not\in B$. We may assume that $n - k - 1 \not\in A$. Let $A' = A \cup \{n-1\}$ and $B' = B \cup \{n\}$. By the definition of $f$, we have $A', B' \in \A$. However, we also have $A' \cap B' = \emptyset$, which contradicts the assumption that $\A$ is intersecting. Therefore, $\A^*(n)$ is intersecting. Since $\A$ is intersecting and each set in $\A$ is $k$-separated, it easily follows that $\C$ is intersecting, as claimed.  

We have $\C \subseteq [n - k - 1]_k^{(r - 1)}$. By the induction hypothesis, \[ |\C| \leq \binom{n - k - 1 - k(r - 1) - 1}{(r - 1) - 1} = \binom{n - kr - 2}{r - 2}.\]

Let $\A^*_{n} = \{A \in \A^* \colon n \in A\}$, let $\C^* = \A^*_{n} \cup \bigcup_{i = 1}^k \A^*_{i, n - k - 1 + i}$, and let $\D = \A^* \setminus \C^*$. Then, $\D \subseteq [n - 1]_k^{(r)}$. Since $\A$ is intersecting, $\D$ is intersecting. By the induction hypothesis, 
\[|\D| \leq \binom{n - kr - 2}{r - 1}.\]

For each $A \in \A^*(n)$, $A$ is $k$-separated, and hence $A \cap [k] = \emptyset$.  Similarly, for each $i \in [k]$ and each $A \in \A^*_{i, n - k  - 1 + i}(n - k - 1 + i)$, we have $j \not\in A$ for each $j \in [k] \setminus \{i\}$. 
Thus, \[|\C| = |\A^*(n)| + \sum_{i = 1}^k |\A^*_{i, n - k - 1 + i}(n - k - 1 + i)| = |\C^*|.\]
 We have
 \[|\A| = |\A^*| =  |\C^*| + |\D| \leq \binom{n - kr - 2}{r - 2} + \binom{n - kr - 2}{r - 1} = \binom{n - kr - 1}{r - 1},\] as required. 
\end{proof}

\section*{Acknowledgements}
Peter Borg was supported by grant MATRP14-20 of the University of Malta. 
Carl Feghali was supported by grant 19-21082S of the Czech Science Foundation.

\bibliography{bibliography}{}
\bibliographystyle{abbrv}

\end{document}